\documentclass[12pt]{amsart}

      \usepackage{amssymb}
      \usepackage[english]{babel}
      \usepackage{graphicx}
      \usepackage[applemac]{inputenc}
      \usepackage{float}
      \usepackage{wrapfig}
      \usepackage[a4paper=true,%
      breaklinks=true,%
      colorlinks=true,%
      linkcolor={blue},%
      citecolor={red},%
      pdftitle={Some bounds and limits in the theory of Riemann's zeta function},%
      pdfauthor={J. Arias de Reyna, J. van de Lune},%
      pdfkeywords={snapshot},%
      pdfstartview=FitH,
]{hyperref}
      \usepackage{listings}
      \theoremstyle{plain}
      \newtheorem{theorem}{Theorem}[section]
      \newtheorem{lemma}[theorem]{Lemma}
      \newtheorem{proposition}[theorem]{Proposition}

      \theoremstyle{definition}

      \theoremstyle{remark}

      \newcommand{\R}{{\mathbb R}}
      \newcommand{\C}{{\mathbb C}}
      \newcommand{\Q}{{\mathbb Q}}
      \newcommand{\N}{{\mathbb N}}
      \newcommand{\Z}{{\mathbb Z}}

      \newfont{\cmbsy}{cmbsy10}
      \newfont{\cmmib}{cmmib10}
      \newcommand{\Orden}{\mathop{\hbox{\cmbsy O}}}

      \def\Real{\mathrm{Re\,}}
      \def\Imag{\mathrm{Im\,}}

      \makeatletter
      \def\@setcopyright{}
      \def\serieslogo@{}
      \makeatother
   
\begin{document}
%


   
   \author{Juan Arias de Reyna$\,\,$}

   \address{Facultad de Matem\'aticas, 
   Universidad de Sevilla, \newline
   Apdo.~1160, 41080-Sevilla, Spain}
   \email{arias@us.es}


   \author{$\,\,$Jan van de Lune }
   \address{Langebuorren 49, 9074 CH Hallum, The Netherlands\newline (Formerly at
   CWI, Amsterdam ) }
   \email{j.vandelune@hccnet.nl}



   \title[Some bounds in the theory of Riemann's zeta function]
   {Some bounds and limits in the theory of Riemann's zeta function}









\maketitle

\begin{abstract}

For any real $a>0$ we determine the supremum of the real $\sigma$ such that 
$\zeta(\sigma+it)=a$ for some real $t$. For $0 < a < 1$, $a = 1$, and 
$a > 1$ the results turn out to be quite different.

We also determine the supremum $E$ of the real parts of the `turning points', that is points 
$\sigma+it$ 
where a curve $\Imag\zeta(\sigma+it)=0$ has a vertical tangent.  
This 
supremum $E$ (also considered by Titchmarsh)  coincides with the supremum 
of the real $\sigma$ such that 
$\zeta'(\sigma+it)=0$ for some real $t$.

We find a surprising connection between the three indicated problems: $\zeta(s)=1$, 
$\zeta'(s)=0$ and turning points of $\zeta(s)$. The almost extremal values for these three 
problems appear to be located at approximately  the same height.  
\end{abstract}


   \section{Introduction.}
   
   In this paper we study various bounds and limits related to the values of  
   Riemann's $\zeta(s)=\zeta(\sigma+it)$ with $s$ in the half-plane
   $\sigma>1$.  
   For example, in Titchmarsh \cite[Theorem 11.5(C)]{T2}
   it is shown 
   that $E:=$ the supremum of all $\sigma$ such that 
   $\zeta'(\sigma+it)=0$ for some $t\in\R$,
   satisfies $2<E<3$.
   Also, one of us \cite{L} proved that $\sigma_0:=$
   the unique solution to the equation 
   $\sum_p\arcsin(p^{-\sigma})=\frac{\pi}{2}$,  
   is the supremum of all $\sigma$ such that $\Real\zeta(\sigma+it)<0$ for 
   some  $t\in\R$ and  $\Real\zeta(\sigma_0+it)>0$ for all $t\in\R$. 
   
   In \cite{NaVW} and \cite{NaVW2} we encounter the question of the 
   supremum $\sigma(1)$ of  $\Real(s)$ for the solutions of $\zeta(s)=1$.  In Sections
   \ref{sectiona} and \ref{section1} we will
   solve this problem and also answer the same question for the solutions of 
   $\zeta(s)=a$ for any given  $a>0$. 
   
   In Section \ref{seccionT}  we  give a more direct proof 
   of Theorem  11.5(C) of Titchmarsh.  
   
   Our method is constructive so that it allowed us to find explicit 
   roots of $\zeta(s)=1$ with $\sigma$ near the extremal value $\sigma(1)$ (
   by means of the Lenstra–Lenstra–Lovász lattice basis reduction algorithm ), and 
   analogously solutions of $\zeta'(s)=0$ with $\Real(s)$ near $E$. We also found
   a relation between  the two problems: 
   Near every almost-extremal solution for $\zeta(s)=1$ there is one for 
   $\zeta'(\rho)=0$ with $\rho-s\approx E-\sigma(1)$ ($\,$see Section \ref{conn} for 
   a more precise formulation$\,$).
   
   In Section \ref{section7} we will discuss some similar aspects 
   of general Dirichlet functions $L(s,\chi)$.

   There are two types of curves $\Imag\zeta(\sigma+it)=0$. One kind (the 
   $I_1$ curves) is crossing the halfplane $\sigma>0$ more or less horizontally
   whereas the other kind (the $I_2$ curves) has the form of a loop. These 
   loops do not stick out arbitrarily far to the right. 
   In Section \ref{S9} we determine exactly the limit of the $I_2$ curves 
   $\Imag\zeta(\sigma+it)=0$. This problem was also mentioned in \cite{L}.
   
   The somewhat surprising fact is that this limit of the  $I_2$ curves 
   is equal to the limit $E$ of the zeros of $\zeta'(s)$ considered in Theorem 
   11.5 (C) of Titchmarsh.


   \section{The key lemmas.}
   We will use the following
   \begin{lemma} \label{KH}
   There exists a sequence of real numbers $(t_k)$  such 
   that 
   \begin{displaymath}
   \lim_{k\to\infty} \zeta(s+it_k)=\frac{2^s-1}{2^s+1}\zeta(s)
   \end{displaymath}
   uniformly on compact sets of the half plane $\sigma>1$.
   \end{lemma}
   
   \begin{proof}
   Since the numbers $\log p_n$ are linearly independent over $\Q$, there are ( by 
   Kronecker's theorem \cite[Theorem 7.9, p.~150]{A} )  for each positive 
   integer $N$ and
   any $\eta>0$
   a real number $t$ and integers $g_1$, \dots, $g_N$ such that 
   \begin{displaymath}
   |-t\log 2-\pi+2\pi g_1|<\eta, \quad 
   |-t\log p_j+2\pi g_j|<\eta, \quad 2\le j\le N
   \end{displaymath}
   where $p_n$ denotes the $n$-th prime number.
   
   Taking $\eta$ small enough we may obtain in this way a real $t$ such that
   \begin{displaymath}
   |2^{-it}+1|<\varepsilon,\quad |p_j^{-it}-1|<\varepsilon,\quad 2\le j\le N.
   \end{displaymath}
   Repeating this construction we obtain a sequence of real numbers $(t_k)$
   such that 
   \begin{displaymath}
   \lim_{k\to\infty} 2^{-it_k}=-1,\quad 
   \lim_{k\to\infty} p^{-it_k}=1\quad \makebox{for any  odd prime $p$}.
   \end{displaymath}
   
   Now we prove that any such sequence satisfies the Lemma.  For any natural
   number $n$ let $\nu(n)$ be the exponent of $2$ in the prime factorization 
   of $n$.
   Let $n=2^{\nu(n)}q_1^{a_1}\cdots q_r^{a_r}$ be the prime factorization of $n$.
   Then we will have $n^{-it_k}\to (-1)^{\nu(n)}$, 
   and as we will show
   \begin{displaymath}
   \lim_{k\to\infty}\zeta(s+it_k) =\sum_{n=1}^\infty \frac{(-1)^{\nu(n)}}{n^s}
   \quad\makebox{uniformly for $\sigma\ge a>1$}.
   \end{displaymath}
   Given $a>1$ and $\varepsilon>0$ we first  determine $N$ such that
   \begin{displaymath}
   \sum_{n=N+1}^\infty\frac{1}{n^{a}}<\varepsilon.
   \end{displaymath}
   For $1\le n\le N$ we then have $n^{-it_k}\to(-1)^{\nu(n)}$, so that 
   there exists a $K$ such that 
   \begin{displaymath}
   |n^{-it_k}-(-1)^{\nu(n)}|<\frac{\varepsilon}{N},\qquad 1\le n\le N, \quad
   k\ge K.
   \end{displaymath}
   For $\Real(s)=\sigma\ge a$ and $k>K$ we will then have
   \begin{multline*}
   \Bigl|\zeta(s+it_k)-\sum_{n=1}^\infty \frac{(-1)^{\nu(n)}}{n^s}\Bigr|\le 
   \Bigl|\sum_{n=1}^\infty \frac{n^{-it_k}-(-1)^{\nu(n)}}{n^s}\Bigr|\le \\
   \le
   \sum_{n=1}^N|n^{-it_k}-(-1)^{\nu(n)}|n^{-a}+
   2\sum_{n=N+1}^\infty n^{-a}\le 3\varepsilon.
   \end{multline*}
   Finally we check whether
   \begin{multline*}
   \sum_{n=1}^\infty \frac{(-1)^{\nu(n)}}{n^s}=
   \Bigl(\sum_{j=0}^\infty\frac{(-1)^j}{2^{js}}\Bigr)
   \Bigl(\sum_{k=0}^\infty\frac{1}{(2k+1)^s}\Bigr)=\\
   =\Bigl(1+\frac{1}{2^s}\Bigr)^{-1}\prod_{p\ge 3}\Bigl(1-\frac{1}{p^s}\Bigr)^{-1}
   =\frac{1-\frac{1}{2^s}}{1+\frac{1}{2^s}}\zeta(s)
   =\frac{2^s-1}{2^s+1}\zeta(s).
   \end{multline*}
   \end{proof}

   \begin{lemma} \label{KH2}
   There exists a sequence of real numbers $(t_k)$  such 
   that 
   \begin{displaymath}
   \lim_{k\to\infty} \zeta(s+it_k)=\frac{\zeta(2s)}{\zeta(s)}
   \end{displaymath}
   uniformly on compact sets of the half plane $\sigma>1$.
   \end{lemma}
   
   \begin{proof}
   The proof is similar to that of the previous Lemma. Applying
   Kronecker's theorem we get a sequence of real numbers $(t_k)$ such that
   \begin{displaymath}
   \lim_{k\to\infty} p^{-it_k}=-1 \qquad \text{for all primes $p$}.
   \end{displaymath}
   Similarly as in the proof of Lemma \ref{KH} we obtain 
   \begin{displaymath}
   \lim_{k\to\infty}\zeta(s+it_k) =\sum_{n=1}^\infty \frac{(-1)^{\Omega(n)}}{n^s}
   \quad\text{uniformly for $\sigma\ge\sigma_0>1$}
   \end{displaymath}   
   where $\Omega(n)$ is the total number of prime factors of $n$ counting 
   multiplicities. 
   It is well known that this series is equal to $\frac{\zeta(2s)}{\zeta(s)}$
   (see Titchmarsh \cite[formula (1.2.11)]{T2}).
   \end{proof}
   
   To apply these lemmas we will use a  theorem of Hurwitz
   (see \cite[Theorem 3.45, p.~119]{T1} or 
   \cite[Theorem 4.10d and Corollary 4.10e, p.~282--283]{H}). 
   We will use it in the following form:
   \begin{theorem}[(Hurwitz)]\label{H}
   Assume that a sequence $(f_n)$ of holomorphic functions on a region $\Omega$
   converges uniformly on compact sets of $\Omega$ to the function $f$ which has
   an isolated  zero $a\in\Omega$. Then for $n\ge n_0$ the functions $f_n$ have a 
   zero $a_n\in\Omega$ such that $\lim_n a_n = a$.
   \end{theorem}

   \section{The bound for $\zeta(s)=a$ ($>0$)  with $a\ne1$.}\label{sectiona}
   
   For a positive real number $a$ let $\sigma(a)$ denote the supremum 
   of all real $\sigma$
   such that $\zeta(\sigma+it)=a$ for some $t\in\R$.
   
   \begin{theorem}
   Let $a$ be $>0$ but $\ne1$. 
   If $a>1$ then  $\sigma(a)$ is the unique solution of 
   $\zeta(\sigma)=a$ with $\sigma>1$.
   If $0<a<1$ then $\sigma(a)$ is the unique solution of $\frac{\zeta(2\sigma)}
   {\zeta(\sigma)}=a$ with $\sigma>1$. 
   \end{theorem}
   
   \begin{proof}
   It will be convenient to 
   define $\sigma_a$ as the ( unique ) solution of the equations considered in the
   theorem. 
    
   The case $a>1$. It is easily seen that in this case we have $\sigma(a)=\sigma_a$. 
   
   In  the case $0<a<1$ we consider  a solution to $\zeta(s)=a$. Then 
   \begin{displaymath}
   a=|\zeta(s)|=\prod_p\frac{1}{\left|1-\frac{1}{p^s}\right|}\ge
   \prod_p\frac{1}{1+\frac{1}{p^\sigma}}=\frac{\zeta(2\sigma)}{\zeta(\sigma)},
   \qquad (\sigma>1).
   \end{displaymath}
   It is clear from the last equality that $\frac{\zeta(2\sigma)}{\zeta(\sigma)}$ 
   is strictly increasing ( for $\sigma>1$ ) from $0$ to $1$. Hence, there 
   exists a unique solution $\sigma_a$ to the equation 
   $a=\frac{\zeta(2\sigma)}{\zeta(\sigma)}$. The inequality 
   $a\ge\frac{\zeta(2\sigma)}{\zeta(\sigma)}$ is then equivalent to 
   $\sigma\le \sigma_a$. Taking the supremum of $\sigma$ for all solutions 
   of $\zeta(s)=a$ we 
   obtain $\sigma(a)\le \sigma_a$. 
   
   To prove the converse we apply Lemma \ref{KH2}: There  exists a sequence
   of real numbers $(t_k)$ such that $\zeta(s+it_k)-a$ converges uniformly on compact 
   sets of $\sigma>1$ to the function $\frac{\zeta(2s)}{\zeta(s)}-a$. 
   The limit function has a zero at $s=\sigma_a$. So, by Hurwitz's theorem  
   $\sigma_a$ is a limit point of zeros $b_{k}$  ($k\ge k_0$) of  
   $\zeta(s+it_{k})-a$. 
   
   Therefore $\zeta(b_{k}+it_{k})-a=0$ and $\lim_k b_{k}=\sigma_a$. 
   For $s_k:=b_{k}+it_{k}$ we have $\zeta(s_k)=a$ and 
   \begin{displaymath}
   \lim_k\Real(s_k)=\lim_k\Real(b_{k})=\Real(\lim_k b_{k})= \sigma_a.
   \end{displaymath}
   It follows that 
   \begin{displaymath}
   \sigma(a)=\sup \{\sigma: \zeta(s)=a\}\ge \lim_k\Real(s_k)=\sigma_a.
   \end{displaymath}
   Therefore $\sigma(a)=\sigma_a$, proving our theorem.

   \end{proof}

   \section{The bound for $\zeta(s)=1$.}\label{section1}
   
   \begin{theorem}\label{Tzeta1}
   The supremum  $\sigma(1)$  of all real $\sigma$ such that 
   $\zeta(\sigma+it)=1$ for some
   value of $t\in\R$, is equal to the unique solution $\sigma>1$ of the equation
   \begin{equation}\label{eq1}
   \zeta(\sigma)=\frac{2^\sigma+1}{2^\sigma-1}.
   \end{equation}
   Numerically we have
   \begin{displaymath}
   \sigma(1) = 1.94010\,16837\,43625\,28601\,74693\,90525\,54887\,
   82302\,47607\dots
   \end{displaymath}
   \end{theorem}
   
   \begin{proof}
   Assume that $\zeta(s)=1$ with $\Real(s)=\sigma>1$. Then by the Euler product 
   formula
   \begin{displaymath}
   1-\frac{1}{2^s}=\prod_{p\ge3}\Bigl(1-\frac{1}{p^s}\Bigr)^{-1}=
   \sum_{k=1}^\infty\frac{1}{(2k-1)^s}
   \end{displaymath}
   or
   \begin{displaymath}
   -1=\sum_{k=2}^\infty \Bigl(\frac{2}{2k-1}\Bigr)^s.
   \end{displaymath}
   Therefore
   \begin{displaymath}
   1=\Bigl|\sum_{k=2}^\infty \Bigl(\frac{2}{2k-1}\Bigr)^s\Bigr|
   \le \sum_{k=2}^\infty \Bigl(\frac{2}{2k-1}\Bigr)^\sigma.    
   \end{displaymath}
   Since the right hand side is decreasing in $\sigma$, it follows that 
   there is a unique solution $\sigma_1$  of the equation
   \begin{equation}\label{eq2}
   1=\sum_{k=2}^\infty \Bigl(\frac{2}{2k-1}\Bigr)^\sigma
   =(2^\sigma-1)\zeta(\sigma)-2^\sigma
   \end{equation}
   and that $\sigma\le \sigma_1$. Now observe that \eqref{eq2} is equivalent
   to \eqref{eq1}.
   Therefore, $\zeta(s)=1$ implies 
   $\sigma\le \sigma_1$ which is by definition the solution of  equation  
   \eqref{eq1}.
   Taking the sup over all solutions of $\zeta(s)=1$ we get $\sigma(1)\le \sigma_1$.
   
   For the converse inequality we apply Lemma \ref{KH} to get a sequence of 
   real numbers $(t_k)$ such that 
   \begin{displaymath}
   \lim_k\{\zeta(s+it_k)-1\} = \frac{2^s-1}{2^s+1}\zeta(s)-1
   \end{displaymath}
   uniformly on compact sets of $\sigma>1$. By definition $\sigma_1$ is a zero 
   of the limit function $\frac{2^s-1}{2^s+1}\zeta(s)-1$, so that there exists
   a natural number $n_0$ and a sequence of complex numbers $(z_k)$ such that
   $\zeta(z_k+it_k)-1=0$ and $\lim_k z_k = \sigma_1$. For $s_k:= z_k+it_k$
   we then have $\zeta(s_k)=1$ 
   and $\lim_k\sigma_k=\sigma_1$ ( with $\sigma_k:=\Real(s_k)$ ).
   
   It follows that $\sigma(1)=\sup_{\zeta(s)=1}\Real s\ge\sigma_1$, proving
   the theorem.
   \end{proof}

   \section{The bound for $\zeta'(s)=0$. A new proof of Titchmarsh's Theorem 11.5(C).}
   \label{seccionT}
   
   Theorem 11.5(C) in Titchmarsh \cite{T2} says that there exists a
   constant $E$ between $2$ and $3$, such that $\zeta'(s)\ne0$ for $\sigma>E$, 
   while $\zeta'(s)$ has an infinity of zeros in every strip between $\sigma=1$ 
   and $\sigma = E$.  In this section we give a more direct proof of this
   theorem and determine the precise value of $E$.

   \begin{theorem}\label{T5.1}
   Let $E$ be the unique solution of the equation 
   \begin{equation}\label{1}
   \frac{2^{\sigma+1}}{4^\sigma-1}\log 2=-\frac{\zeta'(\sigma)}{\zeta(\sigma)},
   \qquad ( \sigma>1 ).
   \end{equation}
   Then $\zeta'(s)\ne0$ for $\sigma>E$, while $\zeta'(s)$ has a sequence of 
   zeros $(s_k)$ with  $\lim_k\Real(s_k)=E$.
   
   The  value of this constant is 
   \begin{displaymath}
   E = 2.81301\,40202\,52898\,36752\,72554\,01216\,68696\,38461\,40560\dots   
   \end{displaymath}
   \end{theorem}
   
   \begin{proof}
   Assuming that $\zeta'(s) = 0$  ( for  $\sigma>1$ ) we have
   \begin{displaymath}
   \frac{\zeta'(s)}{\zeta(s)}=\frac{d}{ds}\log\zeta(s)=\frac{d}{ds}
   \sum_p-\log\Bigl(1-\frac{1}{p^s}\Bigr)=-\sum_p\frac{\log p}{p^s-1}
   \end{displaymath}
   so that we may write the equation $\zeta'(s)=0$ as
   \begin{displaymath}
   \sum_p\frac{\log p}{p^s-1}=0
   \end{displaymath}
   or 
   \begin{displaymath}
   -\frac{\log2}{2^s-1}=\sum_{p\ge3}\frac{\log p}{p^s-1}.
   \end{displaymath}
   So, we must necessarily have
   \begin{displaymath}
   \frac{\log2}{2^\sigma+1}\le\Bigl|-\frac{\log2}{2^s-1}\Bigr|=
   \Bigl|\sum_{p\ge3}\frac{\log p}{p^s-1}\Bigr|\le
   \sum_{p\ge3}\frac{\log p}{p^\sigma-1}
   \end{displaymath}
   and we  may write this inequality as 
   \begin{displaymath}
   \log2\le\sum_{p\ge3}(2^\sigma+1)\Bigl(\frac{1}{p^\sigma}+
   \frac{1}{p^{2\sigma}}+\frac{1}{p^{3\sigma}}+\cdots\Bigr)\log p.
   \end{displaymath}
   Since the right hand side is strictly decreasing in $\sigma$ this is equivalent to 
   $\sigma\le E:=$ the unique solution of the equation
   \begin{displaymath}
   \frac{\log2}{2^{\sigma}+1}+\frac{\log2}{2^{\sigma}-1}= 
   \sum_{p\ge2}\frac{\log p}{p^{\sigma}-1}
   \end{displaymath}
   which is equivalent to \eqref{1}.
   
   This proves that there is no zero of $\zeta'(s)$ with $\sigma>E$. 
   
   Now we must find a sequence of complex numbers $(s_k)$ with 
   \hbox{$\zeta'(s_k)=0$} and $\lim_k\Real(s_k)=E$.
   
   By Lemma \ref{KH} $\zeta'(s+it_k)$ converges uniformly on compact sets of 
   $\sigma>1$ to the function 
   \begin{displaymath}
   \frac{d}{ds}\frac{2^s-1}{2^s+1}\zeta(s)=\Bigl(\frac{2^{s+1}}{4^s-1}\log 2
   +\frac{\zeta'(s)}{\zeta(s)}\Bigr)\cdot \frac{2^s-1}{2^s+1}\zeta(s).
   \end{displaymath}
   This function has a  zero at $s=E$ (see equation \eqref{1}), 
   so that by  Hurwitz's theorem, there exist for $k\ge k_0$ 
   numbers $z_{k}$ such that $z_{k}\to E$ and
   $\zeta'(z_{k}+it_{k})=0$.  Taking $s_k =z_{k}+it_{k}$ we will have
   $\zeta'(s_k)=0$ and 
   \begin{displaymath}
   \lim_k\Real(s_k)=\lim_k \Real(z_{k}+it_{k})=
   \lim_k\Real(z_{k})=E
   \end{displaymath}
   as we wanted to show.

   With Mathematica we found that the solution to equation \eqref{1} is  
   approximately the number 
   given in the theorem.
   \end{proof}

   \section{The connection between $\zeta(s)=1$ and $\zeta'(s)=0$.}\label{conn}
    
   We have seen that to get points with $\zeta(s)=1$ and $\sigma$ near $\sigma(1)$,
   and points $\rho$ with $\zeta'(\rho)=0$ and $\Real\rho$ near $E$, we have applied
   in both cases Lemma \ref{KH}.  The limit function 
   $f(s):=\frac{2^s-1}{2^s+1}\zeta(s)$
   satisfies  $f(\sigma(1))=1$ and $f'(E)=0$. Hence,  from the 
   approximate function $\zeta(s+it_k)$ we may obtain simultaneously  
   points $s$ and $\rho$ 
   with $\zeta(s)=1$ and $\zeta'(\rho)=0$ and more or less to the same height $t_k$.
  
   We will say that a
   sequence of complex numbers $(s_n)$ is \emph{almost extremal} for 
   $\zeta(s)=1$ if $\zeta(s_n)=1$ and $\lim_n\Real(s_n)=\sigma(1)$. Analogously
   $(\rho_n)$ is said to be \emph{almost extremal} for $\zeta'(s)=0$ if
   $\zeta'(\rho_n)=0$ and $\lim_n\Real(\rho_n)=E$. 
   
   First we prove that an almost extremal sequence is related to
   the situation of Lemma \ref{KH}.
   \begin{theorem}\label{sTot}
   (a) If $(s_n)$ is an almost extremal sequence for $\zeta(s)=1$, then $t_n:=
   \Imag(s_n)$ satisfies
   \begin{equation}\label{pcont}
   \lim_{n\to\infty} 2^{-it_n}=-1,\quad 
   \lim_{n\to\infty} p^{-it_n}=1\quad \text{for every  odd prime $p$}.
   \end{equation}
   (b) If $(\rho_n)$ is an almost extremal sequence for $\zeta'(s)=0$, then $t_n:=
   \Imag(\rho_n)$ also satisfies \eqref{pcont}.
   \end{theorem}
   
   \begin{proof}
   (a) Let $s_n=\sigma_n+it_n$. Since $\lim_n\sigma_n=\sigma(1)>1$ we may assume that 
   $\sigma_n>1$ for all $n$. 

   As in the proof of Theorem \ref{Tzeta1}  the equation $\zeta(s_n)=1$ 
   may be written as 
   \begin{displaymath}
   -1=\sum_{k=2}^\infty \Bigl(\frac{2}{2k-1}\Bigr)^{\sigma_n+it_n}.
   \end{displaymath}
   Since $\lim_n\sigma_n=\sigma(1)$, we see that  $\sigma_n$ converges
   to the unique solution to the equation
   \begin{displaymath}
   1=\sum_{k=2}^\infty \Bigl(\frac{2}{2k-1}\Bigr)^{\sigma}.
   \end{displaymath}
   Therefore
   \begin{displaymath}
   \sum_{k=2}^\infty \Bigl(\frac{2}{2k-1}\Bigr)^{\sigma_n+it_n}=-
   \sum_{k=2}^\infty \Bigl(\frac{2}{2k-1}\Bigr)^{\sigma(1)}
   \end{displaymath}
   so that, for all $n\in\N$ we have
   \begin{equation}\label{uno1}
   \sum_{k=2}^\infty\Bigl(\frac{2}{2k-1}\Bigr)^{\sigma(1)}\Bigl(1+
   \Bigl(\frac{2}{2k-1}\Bigr)^{\sigma_n-\sigma(1)+it_n}\Bigr)=0.
   \end{equation}
   We now prove that for each $k\ge2$ we must have
   \begin{equation}\label{dos2}
   \lim_n \Bigl(\frac{2}{2k-1}\Bigr)^{it_n}=-1.
   \end{equation}

   We proceed by contradiction and assume that \eqref{dos2} is not true  
   for some $k_0$.  Since 
   the absolute value of $\bigl(\frac{2}{2k-1}\bigr)^{it_n}$ is $1$, 
   there must exist a subsequence $n_j$ such that 
   \begin{displaymath}
   \lim_{j} \Bigl(\frac{2}{2k_0-1}\Bigr)^{it_{n_j}}= a_{k_0}\ne -1, \qquad 
   |a_{k_0}|=1.
   \end{displaymath}
   By a diagonal argument we may assume that for this subsequence we 
   also have the limits
   \begin{displaymath}
   \lim_{j} \Bigl(\frac{2}{2k-1}\Bigr)^{it_{n_j}}= a_{k}, 
   \qquad |a_k|=1,\quad k\ne k_0.
   \end{displaymath}
   Now consider the equation \eqref{uno1} for $n=n_j$ and take the limit
   for $j\to\infty$. Interchanging limit and sum we then obtain 
   \begin{displaymath}
   \sum_{k=2}^\infty\Bigl(\frac{2}{2k-1}\Bigr)^{\sigma(1)}(1+
   a_k)=0.
   \end{displaymath}
   Now take real parts in this equation.
   Since $\Real(1+a_k)\ge0$ but $\Real(1+a_{k_0})>0$ we get a contradiction, 
   proving \eqref{dos2}.

   Hence, for any  $k$  we have \eqref{dos2}.  
   Now if $p$ is an odd prime we have $p=2k+1$ and $p^2=2m+1$ so that
   \begin{displaymath}
   \lim_n \Bigl(\frac{2}{p}\Bigr)^{it_n}=-1, \quad
   \lim_n \Bigl(\frac{2}{p^2}\Bigr)^{it_n}=-1.
   \end{displaymath}
   Hence 
   \begin{displaymath}
   \lim_n p^{it_n}=\Bigl(\frac{2}{p}\Bigr)^{it_n}\cdot 
   \Bigl(\frac{2}{p^2}\Bigr)^{-it_n}=1
   \end{displaymath}
   so that
   \begin{displaymath}
   \lim_n 2^{it_n}=\lim_n \Bigl(\frac{2}{p}\Bigr)^{it_n} p^{it_n}=-1.
   \end{displaymath}
   \bigskip
   
   (b) Assume now that $(\rho_n)$ is an almost extremal sequence for $\zeta'(s)=0$. 
   Let $\rho_n=\sigma_n+it_n$. Since $\lim_n\sigma_n=E>1$ we may assume that
   $\sigma_n>1$ for all $n$. 
   
   As in the proof of Theorem \ref{T5.1} we will have
   \begin{displaymath}
   \frac{\log2}{2^{\sigma_n}+1}\le\Bigl|-\frac{\log2}{2^{\rho_n}-1}\Bigr|=
   \Bigl|\sum_{p\ge3}\frac{\log p}{p^{\rho_n}-1}\Bigr|\le
   \sum_{p\ge3}\frac{\log p}{p^{\sigma_n}-1}.
   \end{displaymath}
   Since $\lim_n\sigma_n=E$  and $E$ satisfies  equation \eqref{1} we have
   \begin{displaymath}
   \lim_{n\to\infty}\frac{\log2}{2^{\sigma_n}+1}=\lim_{n\to\infty}
   \sum_{p\ge3}\frac{\log p}{p^{\sigma_n}-1}
   \end{displaymath}
   so that
   \begin{equation}\label{inter}
   \lim_{n\to\infty}\Bigl|-\frac{\log2}{2^{\rho_n}-1}\Bigr|=\frac{\log2}{2^{E}+1}
   =\sum_{p\ge3}\frac{\log p}{p^{E}-1}=\lim_{n\to\infty}
   \Bigl|\sum_{p\ge3}\frac{\log p}{p^{\rho_n}-1}\Bigr|.
   \end{equation}
   The first equality in \eqref{inter} implies that 
   $\lim_n|1-2^{\sigma_n+it_n}|=1+2^E$. Let $a$ be a limit point of
   the sequence $(2^{it_n})$. We may choose a sequence $(n_k)$ 
   such that $\lim_k2^{it_{n_k}}=a$. Then 
   $\lim_k|1-2^{\sigma_{n_k}+it_{n_k}}|=|1-2^E a|=1+2^E$. Since $|a|=1$ 
   this is possible only if $a=-1$. Therefore, $(2^{it_n})$, beeing a 
   bounded sequence
   with a unique limit
   point, is convergent and $\lim_n 2^{it_n}=-1$. 
   
   For each odd prime $p$ the sequence $(p^{it_n})$ has $1$ as unique 
   limit point. Indeed, if not, then there is an odd prime $q$ and 
   a sequence $(n_k)$ with 
   \begin{displaymath}
   \lim_k q^{it_{n_k}}= a_q\ne1.
   \end{displaymath} 
   By a diagonal argument we may 
   assume that the limits $\lim_k p^{it_{n_k}}= a_p$ exist for each prime $p$. 
   We will always have $|a_p|=1$. Taking limits in the last equality of 
   \eqref{inter} (for the subsequence $(n_k)$) we obtain
   \begin{displaymath}
   \sum_{p\ge3}\frac{\log p}{p^{E}-1}=
   \Bigl|\sum_{p\ge3}\frac{\log p}{p^E a_p-1}\Bigr|.
   \end{displaymath}
   We have $|p^Ea_p-1|\ge p^E-1$, but the above equality is only possible 
   if we have for all $p$ the equality   $|p^Ea_p-1|= p^E-1$, which
   is in contradiction with our assumption $a_q\ne1$.  
   \end{proof}
   
   Now we can prove the connection between the two problems:
   \begin{theorem}
   Let $(s_n)$ be an almost extremal sequence for $\zeta(s)=1$. Then there
   exists an almost extremal sequence $(\rho_n)$ for $\zeta'(s)=0$ such that
   \begin{displaymath}
   \lim_n\, (\rho_n-s_n) = E-\sigma(1).
   \end{displaymath}
   Analogously if $(\rho_n)$  is an almost extremal sequence for $\zeta'(s)=0$,
   there exists an  almost extremal sequence $(s_n)$ for $\zeta(s)=1$
   satisfying the same condition. 
   \end{theorem}
   
   \begin{proof}
   Let $s_n =\sigma_n+it_n$. By Theorem \ref{sTot} we then have \eqref{pcont}.
   In the proof of Lemma \ref{KH} we have seen that \eqref{pcont} implies
   \begin{displaymath}
   \lim_n\zeta(s+it_n)=\frac{2^s-1}{2^s+1}\zeta(s)\quad \text{uniformly 
   on compact sets of $\sigma>1$}. 
   \end{displaymath}
   It follows that $\zeta'(s+it_n)$ also converges uniformly on compact 
   sets of $\sigma>1$ to the 
   derivative of $f(s):=\frac{2^s-1}{2^s+1}\zeta(s)$. In  
   the proof of Theorem \ref{T5.1} we have seen that $f'(E)=0$.  Hence, by Hurwitz's 
   theorem  for $n\ge n_0$ the function $\zeta'(s+it_n)$ has a zero $s=b_n$
   such that $\lim b_n = E$. Writing $\rho_n:=b_n+it_n$ we have $\zeta'(\rho_n)=0$
   and 
   \begin{displaymath}
   \lim_n \Real(\rho_n)=\lim_n \Real(b_n+it_n)=\lim_n \Real(b_n)=\Real(\lim_n b_n)
   =E.
   \end{displaymath}
   Hence $(\rho_n)$ is almost extremal for $\zeta'(s)=0$ and 
   \begin{displaymath}
   \lim_n(\rho_n-s_n)=\lim_n(b_n-\sigma_n)= E-\sigma(1).
   \end{displaymath}
   The proof for the other case is similar. 
   \end{proof}
   
   \section{Some bounds for Dirichlet $L$-functions.}\label{section7}
   
   Our previous analysis may also be applied to 
   general Dirichlet $L$-functions.  We will give
   two typical examples.  
   
   For the modulus $4$ the non-trivial Dirichlet character is given by 
   $\chi(2n+1)=(-1)^n$,
   $\chi(2n)=0$, so that 
   \begin{displaymath}
   L(s,\chi)=\prod_{p}\Bigl(1-\frac{\chi(p)}{p^s}\Bigr)^{-1}=
   \Bigl(1+\frac{1}{3^s}\Bigr)^{-1}\Bigl(1-\frac{1}{5^s}\Bigr)^{-1}
   \Bigl(1+\frac{1}{7^s}\Bigr)^{-1}\cdots
   \end{displaymath}
   So, the equation $L(s,\chi)=1$ is equivalent to 
   \begin{displaymath}
   \Bigl(1+\frac{1}{3^s}\Bigr)=\Bigl(1-\frac{1}{5^s}\Bigr)^{-1}
   \Bigl(1+\frac{1}{7^s}\Bigr)^{-1}\Bigl(1+\frac{1}{11^s}\Bigr)^{-1}
   \Bigl(1-\frac{1}{13^s}\Bigr)^{-1}\cdots
   \end{displaymath}
   Now ( similarly as in earlier sections ) we let the factor 
   $\left(1+\frac{1}{3^s}\right)$ ``point strictly westward'' and all  other factors
   ``strictly eastward'' (Kronecker's theorem applies here just as well). 
   As in Section \ref{section1}  this leads to  the equation 
   \begin{displaymath}
   \Bigl(1+\frac{1}{3^\sigma}\Bigr)=\Bigl(1-\frac{1}{5^\sigma}\Bigr)^{-1}
   \Bigl(1-\frac{1}{7^\sigma}\Bigr)^{-1}\Bigl(1-\frac{1}{11^\sigma}\Bigr)^{-1}
   \Bigl(1-\frac{1}{13^\sigma}\Bigr)^{-1}\cdots
   \end{displaymath}
   or 
   \begin{displaymath}
   \frac{1+\frac{1}{3^\sigma}}{\left(1-\frac{1}{2^\sigma}\right)
   \left(1-\frac{1}{3^\sigma}\right)}=\zeta(\sigma).
   \end{displaymath}
   (This kind of \emph{trick} also works in the general case. )
   
   Using Mathematica we found that in this case the supremum 
   of all  $\sigma$ such that $L(\sigma+it,\chi)=1$ for some real $t$
   equals 
   \begin{displaymath}
   1.88779\,09267\,08118\,92719\,63215\,42035\,11666\,82234\,70126\dots
   \end{displaymath}
   
   For $n=7$ we find ( for \emph{every charachter} $\chi$ mod 7 ) 
   that $L(s,\chi)=1$ leads
   to the equation 
   \begin{displaymath}
   \frac{1+\frac{1}{2^\sigma}}{\left(1-\frac{1}{2^\sigma}\right)
   \left(1-\frac{1}{7^\sigma}\right)}=\zeta(\sigma)
   \end{displaymath}
   and the bound 
   \begin{displaymath}
   1.83843\,45030\,97314\,94016\,69429\,96760\,82067\,80491\,61315\dots
   \end{displaymath}
   
   For $L(s,\chi)=a$ with $0<a<1$ we let all factors $\left(1-\frac{\chi(p)}{p^s}
   \right)^{-1}$ point ``strictly westward''. This leads to the equation
   \begin{displaymath}
   \prod_p\Bigl(1+\frac{|\chi(p)|}{p^s}\Bigr)^{-1} = a
   \end{displaymath}
   and the \emph{missing factors} are easily supplied.
   For the modulus $4$ and $a=\frac12$ this leads to the equation
   \begin{displaymath}
   \Bigl(1+\frac{1}{2^\sigma}\Bigr)\frac{\zeta(2\sigma)}{\zeta(\sigma)}=
   \frac12
   \end{displaymath}
   and the bound
   \begin{displaymath}
   1.33538\,71957\,45311\,13312\,01066\,99878\,57500\,83328\,78290\dots
   \end{displaymath}
   
   We leave the straightforward  general formulation to the reader.

   \section{Application of the Lenstra–Lenstra–Lovász lattice basis 
   reduction algorithm.}\label{sectionsiete}
   
   For various problems the existence of almost
   extremal sequences $(\sigma_k+it_k)$ depends heavily   
   on the existence of the limits $\lim_k p^{it_k}=: a_p$.  Given a sequence 
   of real numbers $(\theta_j)$, Kronecker's theorem guarantees the existence 
   of a sequence of real numbers $(t_k)$ such that 
   \begin{displaymath}
   \lim_k p_j^{it_k}= e^{i\theta_j},\qquad  (j\in\N).
   \end{displaymath}
   We want to find $t\in\R$  such  that $\sigma+it$ is almost extremal
   for an adequate $\sigma$. To this end, given $n$ we must find $t\in\R$ 
   such that for certain $m_j\in\Z$ 
   \begin{displaymath}
   |t\log p_j-\theta_j-2m_p\pi|<\varepsilon, \qquad 1\le j\le n
   \end{displaymath}
   for some small $\varepsilon$. 
   
   We will use the LLL algorithm similarly as  
   Odlyzko and te Riele \cite{OH} in their disproof of the Mertens conjecture.
   
   Given a basis for a lattice $L$ contained in $\Z^N$, the LLL algorithm  
   yields
   a reduced basis for $L$,  usually consisting of short vectors.
   
   So, we fix $n$, some weights $(w_j)_{j=1}^n$ 
   (in practice we used $w_j=1.15^{40-j}$) and two natural 
   numbers $\nu$ and $r$,
   and
   construct a lattice $L$ in $\Z^{n+2}$  by means of  $n+2$ vectors 
   $v_1$, $v_2$, \dots, $v_n$, $v$ and $v'$
   in $\Z^{n+2}$ ( the method uses lattices in $\Z^N$ ):
   \begin{displaymath}
   \SMALL
   \begin{matrix}
   v_1  = &  (&\lfloor 2\pi w_1\cdot 2^\nu\rfloor, & 0, & 0,&\dots & 0, & 0, & 0)\\
   \noalign{\smallskip}
   v_2  = &  (&0,&\lfloor 2\pi w_2\cdot 2^\nu\rfloor,  & 0, &\dots & 0, & 0, & 0)\\
   \\
   \\
   v_n  = &  (&0, & 0, & 0, &\dots & \lfloor 2\pi w_n\cdot 2^\nu\rfloor, & 0, & 0)\\
   \noalign{\smallskip}
   v    = & (&\lfloor w_1 2^{\nu-r}\lambda_1\rfloor, &
   \lfloor w_2 2^{\nu-r}\lambda_2\rfloor, &
   \lfloor w_3 2^{\nu-r}\lambda_3\rfloor, & \dots & 
   \lfloor w_n 2^{\nu-r}\lambda_n\rfloor, & 0, & 1)\\ \noalign{\smallskip}
   v'  = & ( & -\lfloor w_1\theta_12^\nu\rfloor, &
   -\lfloor w_2\theta_22^\nu\rfloor, &
   -\lfloor w_3\theta_32^\nu\rfloor, & \dots &
   -\lfloor w_n\theta_n2^\nu\rfloor, & 2^\nu n^4, & 0)
   \end{matrix}
   \end{displaymath}
   where we have put $\lambda_j =\log p_j$. 
   
   Applying  the LLL algorithm to these vectors we get a reduced basis 
   $v^*_1$, $v^*_2$, \dots $v^*_{n+2}$ such that at least one of these vectors 
   will have a non-null $(n+1)$-coordinate.  But given that $2^\nu n^4$
   is very large compared with all other entries of the original basis,
   in a reduced basis ( with short vectors ) we do not expect more than one
   large vector.
   Assuming that it is $v^*_1$, its $(n+1)$ coordinate will be 
   $\pm 2^\nu n^4$, and without loss of generality we may assume that it is 
   $2^\nu n^4$. Let $x$ be the last coordinate of $v^*_1$. Then this vector 
   will have coordinate $j$ equal to ( since it is a linear combination of 
   the initial vectors )
   \begin{displaymath}
   x\lfloor w_j2^{\nu-r}\log p_j\rfloor+m_j\lfloor2\pi w_j 2^\nu\rfloor
   -\lfloor w_j \theta_j2^\nu\rfloor
   \end{displaymath}
   for some integers $m_j$. 
   Since it is a reduced basis, we expect this coordinate to be small. Hence also
   the number 
   \begin{displaymath}
   x w_j2^{\nu-r}\log p_j+m_j2\pi w_j 2^\nu
   - w_j \theta_j2^\nu=2^\nu w_j\Bigl(\frac{x}{2^r}\log p_j-\theta_j+2\pi m_j\Bigr)
   \end{displaymath}
   will be small
   and $t = \frac{x}{2^r}$ will have the property we are looking for:
   $t\log p_j-\theta_j+2\pi m_j$ will be small for $1\le j\le n$.

   Figure 1 illustrate    the results obtained.  
   This figure ( and others similar  to it ) is at the 
   origin of our results in  Section \ref{conn}. We were searching for near
   extremal values for the problem $\zeta(s)=1$, and the figure clearly shows
   that we also obtain a near extremal value for the problem $\zeta'(s)=0$.

   The figure represents the rectangle $(-2, 4)\times(h-3,h+3)$ where 
   $h=156326000$.  The solid curves  are those points where $\zeta(s)$ takes
   real values. On the dotted curves $\zeta(s)$ is purely imaginary.  
   For reference we have drawn the lines $\sigma=0$ and $\sigma=1$ 
   limiting the critical strip.
   
   The value $h=156326000$ was given by the LLL algorithm as a candidate
   for a near extreme value of $\zeta(s)=1$.  This is the point labelled $a$.
   In fact $\Real a = 1.907825\dots$ is near the limit $\sigma(1)=1.94010\dots$
   We see also the connected extreme value for $\zeta'(s)=0$. This is the
   point $\rho$ whose real part is also near the corresponding limit value
   $E$.   The role of the point $b$ will be explained in the next Section.

   \section{Bound for the real loops.}\label{S9}
   
   Since $\zeta(s)$ is real for all real $s$, there is no interest in the 
   question of the
   supremum of all $\sigma$ such that $\zeta(\sigma+it)\in\R$ for some $t\in\R$.
   We now focus on the supremum of the real loops.

   Since $u(s):=\Imag\zeta(s)$ is a harmonic function the points where $u(s)=0$
   are arranged in a set of analytic curves.  These curves are of two main 
   types. Some of them traverse the entire plane from $\sigma=-\infty$ to 
   $\sigma=+\infty$ ( in \cite{L} they are called  $I_1$ curves ).  In figure 
   \ref{plot} we have plotted one of these curves. All the other solid curves in this
   figure are $I_2$ curves, they form a loop starting at $\sigma=-\infty$
   and ending again at $\sigma=-\infty$.   Each such $I_2$ curve has a 
   \emph{turning point}, a point on the curve with $\sigma$ maximal.
   
      \begin{figure}[H]
   \begin{center}\label{plot}
   \includegraphics[width = 0.9\textwidth]{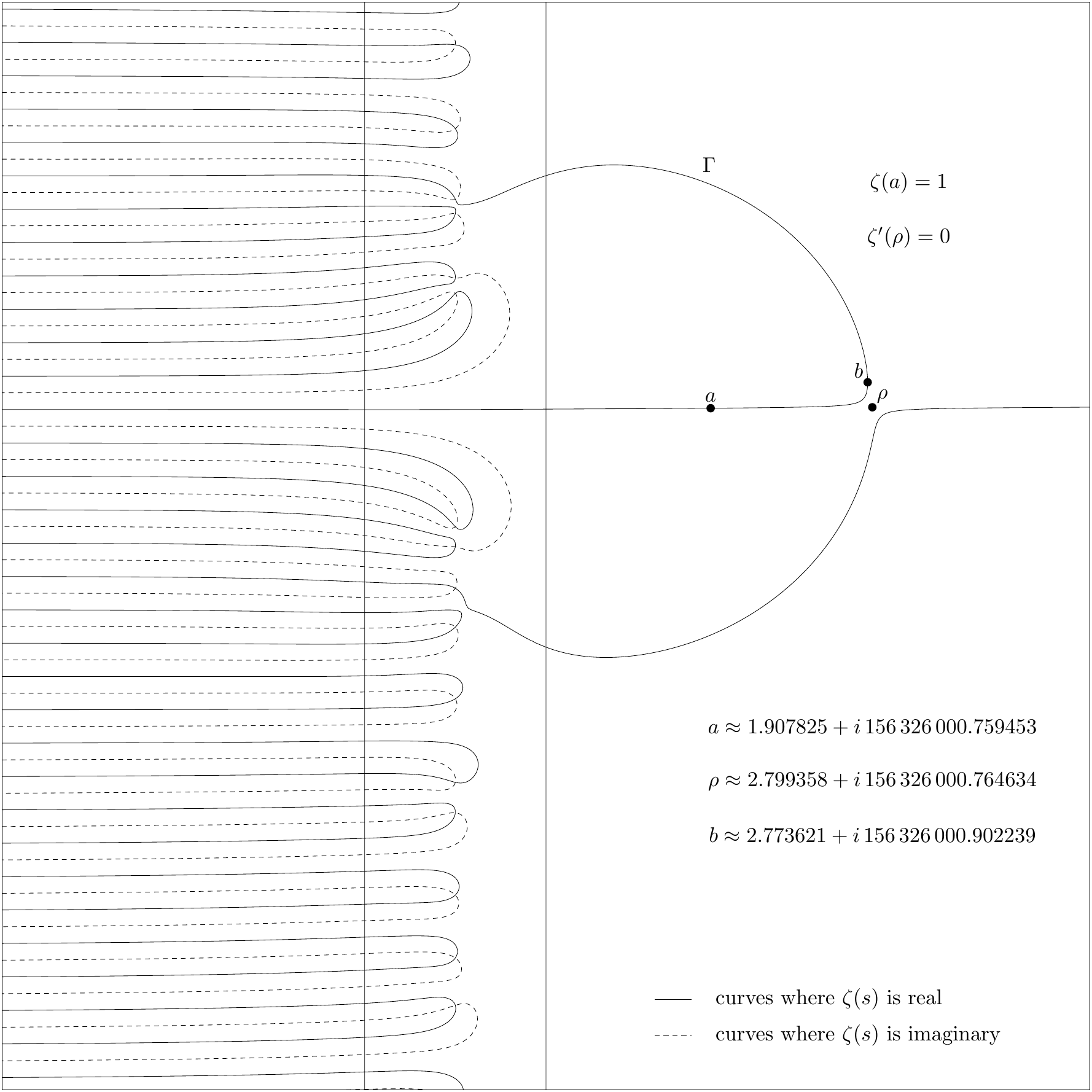}
   \caption{Curves $\Real\zeta(s)=0$ and $\Imag\zeta(s)=0$ near $t=156326000$. }
   \end{center}   
   \end{figure}

   In the case of the curve $\Gamma$ in figure \ref{plot} this is the point
   labelled $b$.  It is easy to see that at these points, since the curve
   $u(\sigma+it)=0$ has a vertical tangent, we must have 
   $u_\sigma(\sigma+it)=0$.  By the Cauchy-Riemann equations this is equivalent 
   to  $\Real\zeta'(\sigma+it)=0$. 
   
   Hence we  define a \emph{turning point} as a point $b=\sigma+it$ 
   such that
   \begin{displaymath}
   \Imag\zeta(b)=0\quad\text{and}\quad \Real\zeta'(b)=0.
   \end{displaymath}
   The first equation says that $b$ is on a real curve (i.~e.~a curve where
   the function $\zeta(s)$ is real), whereas the second equation  means 
   that at the point $b$ the tangent to such a curve is vertical.
   
   The question of the supremum $T$ of all $\sigma$ of turning points of the 
   $I_2$ loops of $\zeta(s)$ was mentioned in \cite{L}.  Here we solve this problem.
   \begin{theorem}\label{T9.1}
   Let $E=2.813014\dots$ be the constant of Theorem \ref{T5.1}. Then each turning 
   point $b=\sigma+it$
   for  $\zeta(s)$ satisfies $\sigma\le E$, and there is 
   a sequence of turning points $(b_k)$ for $\zeta(s)$ with $\lim_k\Real(b_k)=E$.
   \end{theorem}
   
   We will use the following theorem
   \begin{theorem}\label{TLune}
   Let $A$ be the unique solution of the equation
   $$\sum_p\arcsin(p^{-\sigma})=\frac{\pi}{2},\qquad (\sigma>1).$$
   Then $A$ is the supremum of the $\sigma\in\R$ such that there is a 
   $t\in\R$ with $\Real\zeta(\sigma+it)<0$.  For $\sigma=A$ we have 
   $\Real\zeta(\sigma+it)>0$ for all $t\in\R$.
   
   The value of the constant $A$ is 
   \begin{displaymath}
   A = 1.19234\,73371\,86193\,20289\,75044\,27425\,59788\,34011\,
   19230\dots
   \end{displaymath}
   \end{theorem}
   The proof can be found in \cite{L}. The constant $A$ has been computed with high
   precision by R. P.  Brent and J. van de Lune.
   
   We  break the proof of Theorem \ref{T9.1} in several lemmas.
   
   \begin{lemma} The point $\sigma+it$ with $\sigma>A$ is a turning point
   for the function $\zeta(s)$ if and only if  
   \begin{equation}\label{turning}
   \sum_{p}\sum_{k=1}^\infty \frac{1}{k}\frac{\sin(kt\log p)}{p^{k\sigma}}=0
   \quad\text{and}\quad 
   \sum_{p}\sum_{k=1}^\infty \frac{\cos(kt\log p)}{p^{k\sigma}}\log p=0.
   \end{equation}
   \end{lemma}
   
   \begin{proof}
   By Theorem \ref{TLune} for 
   $\sigma > A = 1.192347\dots$  we have \mbox{$\Real \zeta(s) >0$}.  
   In the sequel $\log z$ will be the main branch of the logarithm 
   for $|\arg z|<\pi$, so that $\log \zeta(s)$ is  well defined and 
   analytic for $\sigma > A$.

   In view of $\log z=\log|z|+i\arg z$ it should be clear that, 
   for $\sigma > A$ the two functions $\zeta(s)$ and $\log\zeta(s)$ 
   are real at the same points, so that also the turning
   points of the loops $\Imag\zeta(s)=0$ and $\Imag\log\zeta(s)=0$ 
   are the same.
   \bigskip

   For $s$ real and $>1$ both functions $\zeta(s)$ and $\log\zeta(s)$ 
   are real so that we may write
   \begin{equation}
   \log\zeta(s)=\sum_p\log\Bigl(1-\frac{1}{p^s}\Bigr)^{-1}=
   \sum_p\sum_{k=1}^\infty\frac{1}{k}\frac{1}{p^{ks}},\qquad (\sigma>1)
   \end{equation}
   and this equality is true for $\sigma >A$ by  analytic continuation.  
   \bigskip

   Since the turning points for some function $f(s)$  are defined 
   as the solutions of the system of equations $\Imag f(s) =0$, $\Real f'(s)=0$, 
   the turning points of $\log \zeta(s)$ with $\sigma>A$ are 
   just those points  satisfying equations \eqref{turning}. 
   \end{proof}

   Now we introduce some notations. 
   We may write equations \eqref{turning} in the form
   \begin{equation}\label{turning2}
   \begin{aligned}
   -\sum_{k=1}^\infty \frac{1}{k}\frac{\sin(kt\log 2)}{2^{k\sigma}}
   &=\sum_{p\ge 3}\sum_{k=1}^\infty \frac{1}{k}
   \frac{\sin(kt\log p)}{p^{k\sigma}}\\
   -\sum_{k=1}^\infty \frac{\cos(kt\log 2)}{2^{k\sigma}}\log 2
   &=\sum_{p\ge 3}\sum_{k=1}^\infty 
   \frac{\cos(kt\log p)}{p^{k\sigma}}\log p.
   \end{aligned}
   \end{equation}
   For $\sigma>0$ and $t\in\R$ we now define
   \begin{equation}\label{deffl}
   f(\sigma,t):=\sum_{k=1}^\infty \frac{1}{k}\frac{\sin(kt\log 2)}{2^{k\sigma}}
   \end{equation}
   and 
   \begin{equation}\label{defgl}
   g(\sigma,t):=\frac{\partial}{\partial t} f(\sigma,t)=
   \sum_{k=1}^\infty \frac{\cos(kt\log 2)}{2^{k\sigma}} \log 2.
   \end{equation}
   Note that $f$ and $g$ are periodic functions of $t$ with 
   period $2\pi/\log2$. 

   So, a turning point $\sigma+it$ must satisfy
   \begin{multline*}
   -f(\sigma,t)=\sum_{p\ge 3}\sum_{k=1}^\infty \frac{1}{k}
   \frac{\sin(kt\log p)}{p^{k\sigma}}
   \quad\text{and}\\
   -g(\sigma,t)=\sum_{p\ge 3}\sum_{k=1}^\infty 
   \frac{\cos(kt\log p)}{p^{k\sigma}}\log p.
   \end{multline*}
   \medskip

   We now consider the function  
   \begin{displaymath}
   U(\sigma,t):=2^{2\sigma} f(\sigma,t)^2+
   \Bigl(\frac{2^\sigma}{\log2}\Bigr)^2g(\sigma,t)^2
   \end{displaymath}
   the choice of the coefficients  $2^{2\sigma} $  and  $(2^\sigma/\log2)^2$ 
   being motivated by  ($\,$use \eqref{deffl} and \eqref{defgl}$\,$)
   \begin{multline*}
   \lim_{\sigma\to+\infty}U(\sigma,t)=\\
   \lim_{\sigma\to+\infty}
   \Bigl\{2^{2\sigma}\Bigl(\sum_{k=1}^\infty 
   \frac{1}{k}\frac{\sin(kt\log 2)}{2^{k\sigma}}\Bigr)^2+
   \Bigl(\frac{2^\sigma}{\log2}\Bigr)^2
   \Bigl(\sum_{k=1}^\infty \frac{\cos(kt\log 2)}{2^{k\sigma}} \log 2\Bigr)^2\Bigr\}\\
   =\sin^2(t\log2)+\cos^2(t\log2)=1.
   \end{multline*}

   \begin{lemma}\label{Cauchy}
   Let $a$ and $b$ be arbitrary real numbers. Then there exist 
   real numbers $x$ and $y$ such that
   \begin{displaymath}
   ax+by=(a^2+b^2)^{1/2}\quad \text{and}\quad x^2+y^2=1.
   \end{displaymath}
   \end{lemma}

   \begin{proof} If $a^2+b^2=0$ then $a=b=0$ and we need only 
   take $x$ and $y$ such that $x^2+y^2=1$. 

   If $a^2+b^2\ne0$ then we can take $x=\frac{a}{\sqrt{a^2+b^2}}$ and
   $y=\frac{b}{\sqrt{a^2+b^2}}$.
   \end{proof}

   \begin{lemma}\label{Lturningineq}
   If $\sigma+it$ is a turning point of $\zeta(s)$ with $\sigma>A$, then 
   \begin{displaymath}
   U(\sigma,t)< \Bigl(\frac{2^\sigma}{\log2}\Bigr)^2
   \Bigl(\sum_{p\ge3} \sum_{k=1}^\infty\frac{\log p}{p^{ks}}\Bigr)^2.
   \end{displaymath}
   \end{lemma}

   \begin{proof}
   We apply  Lemma \ref{Cauchy} to
   \begin{displaymath}
   a = -2^{\sigma} f(\sigma,t)
   \quad \text{and} \quad 
   b = -\frac{2^\sigma}{\log2}g(\sigma,t)
   \end{displaymath}
   to get
   \begin{displaymath}
   \Bigl\{2^{2\sigma} f(\sigma,t)^2+\Bigl(\frac{2^\sigma}{\log2}\Bigr)^2
   g(\sigma,t)^2\Bigr\}^{1/2} = -2^{\sigma} x f(\sigma,t)- 
   \frac{2^\sigma}{\log2} y g(\sigma,t)
   \end{displaymath}
   which, by \eqref{turning2}, may be written as
   \begin{multline*}
   U(\sigma,t)^{1/2}  =2^\sigma \sum_{p\ge 3}
   \sum_{k=1}^\infty \frac{x}{k}\frac{\sin(kt\log p)}{p^{k\sigma}}+\\+
   \frac{2^\sigma}{\log2}\sum_{p\ge 3}
   \sum_{k=1}^\infty \frac{y\cos(kt\log p)}{p^{k\sigma}}\log p=
   \\=
   2^\sigma\sum_{p\ge 3}\sum_{k=1}^\infty
   \Bigl(\frac{x}{k}\frac{\sin(kt\log p)}{p^{k\sigma}}+
   \frac{y\log p}{\log2}\frac{\cos(kt\log p)}{p^{k\sigma}}\Bigr).
   \end{multline*}
   Applying the Cauchy-Schwarz inequality to the right hand side 
   we obtain the condition
   \begin{multline}\label{nueve}
   U(\sigma,t)^{1/2} \le \\ \le2^\sigma\sum_{p\ge 3}
   \sum_{k=1}^\infty\Bigl(\frac{x^2}{k^2}+
   \frac{y^2\log^2 p}{\log^22}\Bigr)^{1/2}
   \Bigl(\frac{\sin^2(kt\log p)+\cos^2(kt\log p)}{p^{2k\sigma}}\Bigr)^{1/2}.
   \end{multline}
   Now observe that in \eqref{nueve} $\frac{1}{k^2}<\frac{\log^2p}{\log^22}$ 
   so that
   \begin{displaymath}
   \frac{x^2}{k^2}+\frac{y^2\log^2 p}{\log^22}< 
   \frac{(x^2+y^2)\log^2 p}{\log^22}\le 
   \frac{\log^2 p}{\log^22}.
   \end{displaymath}
   Using this we thus obtain the condition
   \begin{displaymath}
   U(\sigma,t)^{1/2}< 2^\sigma\sum_{p\ge 3}
   \sum_{k=1}^\infty\frac{\log p}{\log2}\frac{1}{p^{k\sigma}}
   \end{displaymath}
   or 
   \begin{displaymath}
   U(\sigma,t)<
   \Bigl(\frac{2^\sigma}{\log2}\Bigr)^2
   \Bigl(\sum_{p\ge 3}\sum_{k=1}^\infty
   \frac{\log p}{p^{k\sigma}}\Bigr)^2
   \end{displaymath}
   as we wanted to show.
   \end{proof}

   For $\sigma>1$ we define 
   \begin{equation}\label{defH}
   H(\sigma):=\Bigl(\frac{2^\sigma}{\log2}\Bigr)^2
   \Bigl(\sum_{p\ge3}\sum_{k=1}^\infty\frac{\log p}{p^{k\sigma}}\Bigr)^2.
   \end{equation}

   \begin{lemma}\label{lemau}
   For each $t\in\R$ there exists a largest solution $u(t)$ 
   to the equation in $\sigma$
   \begin{equation}\label{eqdefu}
   U(\sigma,t)=H(\sigma)
   \end{equation}
   and
   \begin{displaymath}
   U(\sigma,t)>H(\sigma), \qquad (\sigma>u(t)).
   \end{displaymath}
   \end{lemma}

   \begin{proof}
   By \eqref{defH} it is easily seen that $H(\sigma)$ is  continuous 
   and strictly decreasing for $\sigma>1$ from $+\infty$ to $0$. 
   In particular 
   \begin{displaymath}
   \lim_{\sigma\to\infty}H(\sigma)=0.
   \end{displaymath}

   Since $U(\sigma,t)$ is continuous for $\sigma>0$ and $t\in\R$, and
   \begin{displaymath}
   \lim_{\sigma\to+\infty} U(\sigma,t)=1
   \end{displaymath}
   we see that for every $t$ the infimum $u(t)$ of the $a$ such that
   $U(\sigma,t)>H(\sigma)$ for $\sigma>a$ exists and is larger than $1$. 

   From this it is clear that $u(t)$ must be a solution of  
   equation \eqref{eqdefu}
   in $\sigma$.
   \end{proof} 

   \begin{lemma}\label{closed}
   We have the closed formulas
   \begin{multline*}
   f(\sigma,t) = \arctan
   \frac{\sin(t\log2)}{2^\sigma-\cos(t\log 2)},\\
   g(\sigma,t)=-\frac{(1-2^\sigma\cos(t\log2))\log2}
   {1+4^\sigma-2^{1+\sigma}\cos(t\log2)}.
   \end{multline*}
   \end{lemma}
   
   \begin{proof}
   The first follows from the identity $f(\sigma,t)=\Imag(\log(1-2^{-s}))$,
   and the second by differentiation.
   \end{proof}

   \begin{lemma}\label{u=E}
   We have $u(\pi/\log2)=E$.
   \end{lemma}

   \begin{proof}
   We have
   \begin{displaymath}
   \sum_{p\ge3}\sum_{k=1}^\infty\frac{\log p}{p^{k\sigma}}=
   \sum_{p}\sum_{k=1}^\infty\frac{\log p}{p^{k\sigma}}-
   \sum_{k=1}^\infty\frac{\log 2}{2^{k\sigma}}=
   -\frac{\zeta'(\sigma)}{\zeta(\sigma)}-\frac{\log2}{2^\sigma-1}
   \end{displaymath}
   so that 
   \begin{displaymath}
   H(\sigma)=\Bigl(\frac{2^\sigma}{\log2}\Bigr)^2
   \Bigl(\frac{\zeta'(\sigma)}{\zeta(\sigma)}+
   \frac{\log2}{2^\sigma-1}\Bigr)^2.
   \end{displaymath}
   By its definition $u(t)$ is the largest solution of the 
   equation $U(\sigma,t)=H(\sigma)$. 

   For $t=\pi/\log2$ we have
   \begin{displaymath}
   f(\sigma,t)=\sum_{k=1}^\infty \frac{1}{k}
   \frac{\sin(kt\log2)}{2^{k\sigma}}=
   \sum_{k=1}^\infty \frac{1}{k}\frac{\sin(k\pi)}{2^{k\sigma}}=0
   \end{displaymath}
   and
   \begin{multline*}
   g(\sigma,t)=\sum_{k=1}^\infty \frac{\cos(kt\log2)}{2^{k\sigma}}\log2=
   \sum_{k=1}^\infty \frac{\cos(k\pi)}{2^{k\sigma}}\log2=\\=
   \sum_{k=1}^\infty \frac{(-1)^k}{2^{k\sigma}}\log2
   =-\frac{\log2}{2^\sigma+1}
   \end{multline*}
   so that  $u(\pi/2)$ satisfies the equation 
   \begin{displaymath}
   \Bigl(\frac{2^\sigma}{\log2}\Bigr)^2
   \Bigl(\frac{\log2}{2^\sigma+1}\Bigr)^2=
   \Bigl(\frac{2^\sigma}{\log2}\Bigr)^2
   \Bigl(\frac{\zeta'(\sigma)}{\zeta(\sigma)}+
   \frac{\log2}{2^\sigma-1}\Bigr)^2.
   \end{displaymath}
   Since $\frac{\zeta'(\sigma)}{\zeta(\sigma)}+
   \frac{\log2}{2^\sigma-1}<0$ this is equivalent to 
   \begin{displaymath}
   \frac{\log2}{2^\sigma+1}=-\frac{\zeta'(\sigma)}{\zeta(\sigma)}-
   \frac{\log2}{2^\sigma-1}
   \end{displaymath}
   or
   \begin{displaymath}
   \frac{2^{\sigma+1}}{4^\sigma-1}\log 2=
   -\frac{\zeta'(\sigma)}{\zeta(\sigma)}.
   \end{displaymath}
   But $E$ is the unique solution of this equation for 
   $\sigma>1$ ( see  Theorem \ref{T5.1} ). 

   Hence $u(\pi/\log2)=E$. 
   \end{proof}

   \begin{lemma}\label{lemaUmH}
   For all $\sigma>1$ and all $t\in\R$ we have
   \begin{equation}\label{Uineq}
   U(\sigma,t)\ge U(\sigma,\pi/\log2).
   \end{equation}
   \end{lemma}

   \begin{proof}
   We have computed $U(\sigma,\pi/\log2)$ in the proof of Lemma \ref{u=E}. 
   Substituting  this value and the definition of $U(\sigma,t)$, 
   \eqref{Uineq} may be written
   \begin{displaymath}
   2^{2\sigma} f(\sigma,t)^2+\Bigl(\frac{2^\sigma}{\log2}\Bigr)^2
   g(\sigma,t)^2 \ge \Bigl(\frac{2^\sigma}{\log2}\Bigr)^2
   \Bigl(\frac{\log2}{2^\sigma+1}\Bigr)^2.
   \end{displaymath}
   In view of Lemma \ref{closed} we thus need to prove
   \begin{multline}\label{B1}
   \arctan^2\Bigl(\frac{\sin(t\log2)}{2^\sigma-\cos(t\log 2)}\Bigr)
   +\Bigl(\frac{(1-2^\sigma\cos(t\log2))}
   {1+4^\sigma-2^{1+\sigma}\cos(t\log2)}\Bigr)^2\ge \\ \ge
   \Bigl(\frac{1}{2^{\sigma}+1}\Bigr)^2.
   \end{multline}
 
   We change notations  putting $t\log2=\varphi$ and $2^{\sigma}=x^{-1}$,  
   so that we have to prove for $0<x<1$ and $0<\varphi<2\pi$
   \begin{equation}\label{A3}
   u(x,\varphi):=\arctan^2\Bigl(\frac{x\sin\varphi}{1-x\cos\varphi}\Bigr)
   +\Bigl(\frac{x(x-\cos\varphi)}{1+x^2-2x\cos\varphi}\Bigr)^2
   \ge \Bigl(\frac{x}{1+x}\Bigr)^2.
   \end{equation}
   The right hand side is the value for $\varphi=\pi$ of the left hand side.

   So, we want to prove that  $u(x,\varphi)$ has an absolute minimum 
   at $\varphi=\pi$.
   It is easy to show that $u(x,\pi-\theta)=u(x,\pi+\theta)$. 
   So,   we only have to prove inequality \eqref{A3} for $0<\varphi<\pi$. 
   We will split the proof in two cases.
   \bigskip

   (1) Proof of \eqref{A3} for  $\frac{\pi}{2}<\varphi<\pi$.

   If we differentiate $u(x,\varphi)$ with respect to $\varphi$ 
   and simplify we  arrive at
   \begin{multline}\label{A4}
   u_\varphi(x,\varphi) = 
   \frac{2x(x-\cos \varphi)}{(1+x^2-2x\cos\varphi)^3}
   \Bigl\{-\arctan\Bigl(\frac{x\sin\varphi}{1-x\cos\varphi}\Bigr)
   \times \\ \times
   (1+x^2-2x\cos\varphi)^2+x(1-x^2)\sin\varphi\Bigr\}.
   \end{multline}
   We will show that $u_\varphi(x,\varphi)<0$ for $\frac{\pi}{2}<\varphi<\pi$, 
   so that \eqref{A3} will follow.

   In this interval $\cos\varphi<0$ and $\sin\varphi>0$. The first factor
   in the right hand side of \eqref{A4} is positive, and we will show 
   that the second is negative. That is we will show that 
   \begin{equation}\label{A5}
   x(1-x^2)\sin\varphi\le \arctan
   \Bigl(\frac{x\sin\varphi}{1-x\cos\varphi}\Bigr) 
   (1+x^2-2x\cos\varphi)^2.
   \end{equation}
   Let 
   \begin{equation}\label{B2}
   \alpha=\arctan\Bigl(\frac{x\sin\varphi}{1-x\cos\varphi}\Bigr),
   \quad \tan\alpha=\frac{x\sin\varphi}{1-x\cos\varphi},
   \end{equation}
   \begin{displaymath}
   \frac{1}{\cos^2\alpha}=
   1+\Bigl(\frac{x\sin\varphi}{1-x\cos\varphi}\Bigr)^2=
   \frac{1+x^2-2x\cos\varphi}{(1-x\cos\varphi)^2},
   \end{displaymath}
   \begin{displaymath}
   \cos^2\alpha=\frac{(1-x\cos\varphi)^2}{1+x^2-2x\cos\varphi},
   \end{displaymath}
   \begin{equation}\label{B3}
   \sin^2\alpha=1-\frac{(1-x\cos\varphi)^2}{1+x^2-2x\cos\varphi}=
   \frac{x^2-x^2\cos^2\varphi}{1+x^2-2x\cos\varphi}=
   \frac{x^2\sin^2\varphi}{1+x^2-2x\cos\varphi}
   \end{equation}
   so that 
   \begin{displaymath}
   \sin\alpha=\frac{x\sin\varphi}{\sqrt{1+x^2-2x\cos\varphi}}
   \end{displaymath}
   (with the sign $+$ since certainly $\alpha\in(0,\pi/2)$, since $\tan\alpha>0$).

   Now we have
   \begin{displaymath}
   1-x^2<1<(1+x^2-2x\cos\varphi)^{3/2}
   \end{displaymath}
   so that
   \begin{displaymath}
   x(1-x^2)\sin\varphi\le x\sin\varphi (1+x^2-2x\cos\varphi)^{3/2}
   \end{displaymath}
   and  
   \begin{displaymath}
   x(1-x^2)\sin\varphi\le \sin\alpha (1+x^2-2x\cos\varphi)^2\le
   \alpha (1+x^2-2x\cos\varphi)^2
   \end{displaymath}
   which  is equivalent to \eqref{A5}.
   \medskip

   (2) Proof of \eqref{A3} for  $0<\varphi<\frac{\pi}{2}$.

   Defining $\alpha$ as in \eqref{B2}, $\sin^2\alpha$ is still given by \eqref{B3}.
   Although in this case we do not know the sign of  $\sin\alpha$, 
   inequality \eqref{A3} will still  follow from
   \begin{equation}\label{A6}
   \frac{x^2\sin^2\varphi}{1+x^2-2x\cos\varphi}+
   \Bigl(\frac{x(x-\cos\varphi)}{1+x^2-2x\cos\varphi}\Bigr)^2
   \ge \Bigl(\frac{x}{1+x}\Bigr)^2
   \end{equation}
   since $\sin^2\alpha<\alpha^2$.

   To prove \eqref{A6} we consider two cases.
   \medskip

   (2a) Proof of \eqref{A6} when $1+x^2-2x\cos\varphi>1$. 

   Then  $(1+x^2-2x\cos\varphi)^2>1+x^2-2x\cos\varphi$, so that 
   \begin{multline*}
   \frac{\sin^2\varphi}{1+x^2-2x\cos\varphi}+
   \frac{(x-\cos\varphi)^2}{(1+x^2-2x\cos\varphi)^2}\ge \\ \ge
   \frac{\sin^2\varphi}{(1+x^2-2x\cos\varphi)^2}+
   \frac{(x-\cos\varphi)^2}{(1+x^2-2x\cos\varphi)^2}=
   \\=\frac{1+x^2-2x\cos\varphi}{(1+x^2-2x\cos\varphi)^2}=
   \frac{1}{1+x^2-2x\cos\varphi}.
   \end{multline*}
   Recall that $0<\varphi<\frac{\pi}{2}$. Then  $-2x\cos\varphi<2x$, 
   so that $1+x^2-2x\cos\varphi<1+x^2+2x=(1+x)^2$, and we obtain
   \begin{displaymath}
   \frac{\sin^2\varphi}{1+x^2-2x\cos\varphi}+
   \frac{(x-\cos\varphi)^2}{(1+x^2-2x\cos\varphi)^2}> 
   \frac{1}{(1+x)^2}.
   \end{displaymath}
   \medskip

   (2b) Proof of \eqref{A6} when $1+x^2-2x\cos\varphi\le1$.
   In this case $(1+x^2-2x\cos\varphi)^2\le 1+x^2-2x\cos\varphi$ so that 
   \begin{multline*}
   \frac{\sin^2\varphi}{1+x^2-2x\cos\varphi}+
   \frac{(x-\cos\varphi)^2}{(1+x^2-2x\cos\varphi)^2}\ge \\ \ge
   \frac{\sin^2\varphi}{(1+x^2-2x\cos\varphi)}+
   \frac{(x-\cos\varphi)^2}{(1+x^2-2x\cos\varphi)}=\\
   =\frac{1+x^2-2x\cos\varphi}{1+x^2-2x\cos\varphi}=1>\frac{1}{(1+x)^2}.
   \end{multline*}
   \end{proof}

   \begin{lemma}\label{Lmax}
   For each $t\in\R$ we have $u(t)\le u(\pi/\log2)$.
   \end{lemma}

   \begin{proof}
   By Lemma \ref{lemau}
   \begin{displaymath}
   U(\sigma,\pi/\log2)>H(\sigma) \quad\text{for}\quad 
   \sigma>u(\pi/\log2)
   \end{displaymath}
   and  by Lemma \ref{lemaUmH}
   \begin{displaymath}
   U(\sigma,t)\ge U(\sigma,\pi/\log2).
   \end{displaymath}
   It follows that 
   \begin{displaymath}
   U(\sigma,t)>H(\sigma),\qquad (\sigma>u(\pi/\log2)).
   \end{displaymath}
   By definition $U(\sigma,t)>H(\sigma)$ is not true for $\sigma=u(t)$, 
   and it follows that $u(t)\le u(\pi/\log2)$.
   \end{proof}

   \begin{proof}[Proof of the first half of Theorem \ref{T9.1}]
   Let $\sigma+it$ be a turning point for $\zeta(s)$.  
   It is clear that  $\sigma\le A=1.192\dots$ implies  $\sigma<E=2.813\dots$.
   For $\sigma>A$,  by Lemma \ref{Lturningineq} we will have 
   \begin{displaymath}
   U(\sigma,t)< H(\sigma)
   \end{displaymath}
   so that Lemma \ref{lemau} implies that 
   \begin{displaymath}
   \sigma< u(t).
   \end{displaymath}
   By Lemma \ref{Lmax}
   \begin{displaymath}
   u(t)\le u(\pi/\log2)
   \end{displaymath}
   and by Lemma \ref{u=E}
   \begin{displaymath}
   u(\pi/\log2)=E.
   \end{displaymath}
   It follows that $\sigma< E$.

   Therefore, the supremum $T$ of the real parts of the turning points  
   is less than or equal  to $E$.
   We have even proved a little more: On the line $\sigma=E$ there is 
   no turning point.
   \end{proof} 
   
   We will now show that there is a sequence $(b_n)$ of turning points for 
   $\zeta(s)$ such that $\lim_n \Real(b_n) = E$. This will end the proof 
   of Theorem \ref{T9.1}.
   \medskip

   By Lemma \ref{KH} there exists a sequence of real numbers $(t_k)$ 
   such that $\zeta(s+it_k)$ converges to $f(s):=\frac{2^s-1}{2^s+1}\zeta(s)$. 
   Since
   \begin{displaymath}
   f(E)=0.9\dots,\quad f'(E)=0,\quad f''(E)=0.07\dots,\quad f'''(E)=-0.17\dots
   \end{displaymath}
   $E$ is a turning point for $f(s)$.

   We are going to show that the functions 
   $\zeta(s+it_k)$ must have a turning point very near to $E$.  
   \medskip

   We prove a slightly more general result. We break the proof in several 
   lemmas.
   \medskip

   Given a holomorphic function $f$ defined on a disc with center at $0$ and
   radius $R$ we define the associated (continuous) function
   \begin{displaymath}
   h(r,\varphi)=\Imag f(re^{i\varphi})+i\Real f'(re^{i\varphi})
   \end{displaymath}
   so that $re^{i\varphi}$ will be a turning point for $f(z)$ if and only if 
   $h(r,\varphi)=0$. 
   \medskip

   For each $0<r<R$ let $\gamma_r$ be the curve  $\varphi\colon[0,2\pi)
   \mapsto h(r,\varphi)$.

   \begin{proposition}\label{P1}
   Let $f(z)=a_0+a_2z^2+a_3z^3+\cdots$ be a holomorphic function 
   on $\Delta(0,R)$  the disc with center  $0$ and radius $R$. 
   Assume that $a_0>0$, $a_2>0$ and $a_3<0$. Then there exists an 
   $r_0>0$ such that for $0<r<r_0$,  the curve $\gamma_r$ does not 
   pass through $z=0$  and the index ($\,$the winding number$\,$)
   of the curve $\gamma_r$ with respect to $0$ is $\omega(\gamma_r,0)=1$. 
   \end{proposition}

   To prove Proposition \ref{P1} we will use some lemmas. 

   \begin{lemma}\label{LET1}
   Let $f$ be as in Proposition \ref{P1} and define 
   \begin{displaymath}
   u(r,\varphi):=\Imag f(re^{i\varphi}), \qquad 
   v(r,\varphi):=\Real f'(re^{i\varphi}).
   \end{displaymath}
   Then there exists $r_0$ such that for $0<r<r_0$,  ($r\to0$)
   \begin{align*}
   u(r,\varphi)&=a_2r^2\sin2\varphi+a_3 r^3\sin3\varphi+\Orden(r^4)\\
   v(r,\varphi)&=2a_2r\cos\varphi+3a_3r^2\cos2\varphi+\Orden(r^3)\\
   u_\varphi(r,\varphi)&=2a_2r^2\cos2\varphi+3a_3r^3\cos3\varphi+\Orden(r^4)\\
   v_\varphi(r,\varphi)&= -2a_2r\sin\varphi-6a_3r^2\sin2\varphi+\Orden(r^3)
   \end{align*}
   where the implicit constants do not depend on $\varphi$.
   \end{lemma}

   \begin{proof}
   Let $f(z)=\sum_{n=0}^\infty a_nz^n$ be the power series of $f$ 
   at $0$, and take $r_0$ less than the radius of convergence.
   Then 
   \begin{displaymath}
   f(z)=a_0+a_2z^2+a_3z^3+\sum_{n=4}^\infty a_n z^n
   \end{displaymath}
   so that 
   \begin{displaymath}
   u(r,\varphi)=a_2r^2\sin2\varphi+a_3 r^3\sin3\varphi+
   \sum_{n=4}^\infty r^n\Imag( a_n e^{in\varphi}) 
   \end{displaymath}
   and 
   \begin{displaymath}
   u_\varphi(r,\varphi)=2a_2r^2\cos2\varphi+3a_3 r^3\cos3\varphi+
   \sum_{n=4}^\infty r^n\Imag (in a_n e^{in\varphi}) 
   \end{displaymath}
   and for $0<r<r_0$ we will have
   \begin{multline*}
   \Bigl|\sum_{n=4}^\infty r^n \Imag(a_n e^{in\varphi})\Bigr|\le r^4
   \sum_{n=4}^\infty |a_n|r_0^{n-4},\\
   \Bigl|\sum_{n=4}^\infty r^n\Imag (in a_n e^{in\varphi})\Bigr|\le 
   r^4\sum_{n=4}^\infty n|a_n|r_0^{n-4}.
   \end{multline*}
   The  last two sums converge and this proves our lemma for 
   $u$ and $u_\varphi$.  For $v$ and $v_\varphi$ the proof is similar.
   \end{proof}

   We divide the interval $[-\frac{\pi}{8}, \frac{15\pi}{8}]$ of 
   length $2\pi$ in $8$  intervals
   \begin{gather*}
   I_1 =[-\pi/8,\pi/8],\quad 
   I_2 =[\pi/8,3\pi/8],\quad
   I_3 =[3\pi/8,5\pi/8],\\
   I_4 =[5\pi/8,7\pi/8],\quad
   I_5 =[7\pi/8,9\pi/8],\quad
   I_6 =[9\pi/8,11\pi/8],\\
   I_7 =[11\pi/8,13\pi/8],\quad
   I_8 =[13\pi/8,15\pi/8].
   \end{gather*}

   \begin{lemma}\label{L9.13}
   There exists an $r_0>0$ such that for $0<r<r_0$
   the function $u$ has exactly four zeros  on $[-\pi/8,15\pi/8]$, 
   denoted by $\alpha_1\in I_1$, $\alpha_3\in I_3$, $\alpha_5\in I_5$ 
   and $\alpha_7\in I_7$, so that $u$ is positive on $(\alpha_1,\alpha_3)$, 
   negative on $(\alpha_3,\alpha_5)$, 
   positive on $(\alpha_5,\alpha_7)$ and 
   negative on $(\alpha_7,\alpha_1+2\pi)$
   \end{lemma}

   \begin{proof}
   By Lemma \ref{LET1}  for $r\to0$
   \begin{displaymath}
   u(r,\varphi)=a_2r^2(\sin2\varphi+\Orden(r)),\quad 
   u_\varphi(r,\varphi)=2a_2r^2(\cos2\varphi+\Orden(r)).
   \end{displaymath}
   On $I_2$ and $I_6$ $\sin2\varphi>2^{-1/2}$, whereas 
   $\sin2\varphi<-2^{-1/2}$ on $I_4$ and $I_8$.
   Then, if we take $r_0$ small enough, $u(r,\varphi)>0$ on $I_2$ and 
   $I_6$, and  $u(r,\varphi)<0$ on $I_4$ and $I_8$  (we only need to 
   take  the $\Orden(r)$ terms  less than $2^{-1/2}$).

   By continuity of $u(r,\varphi)$ this implies that for each 
   $0<r<r_0$ the function  $u(r,\varphi)$ has at least one zero 
   on each of the intervals $I_1$, $I_3$, $I_5$ and $I_7$. 
   But  $\cos2\varphi>2^{-1/2}$ on $I_1$ and $I_5$, and  
   $\cos2\varphi<-2^{-1/2}$ on $I_3$ and $I_7$, so that choosing 
   $r_0$ small enough the sign of $u_\varphi(r,\varphi)$ will be 
   negative on $I_3$ and $I_7$ and  positive on $I_1$ and $I_5$. 
   Therefore on each of these intervals the function
   $u(r,\varphi)$ is monotonic and has only one zero.
   \end{proof}

   There is an analogous result for $v(r,\varphi)$.

   \begin{lemma}
   There exists an $r_0>0$ such that for $0<r<r_0$ 
   the function $v(r,\varphi)$ has exactly two zeros for   
   $\varphi\in[-\pi/8,15\pi/8]$,  denoted by
   $\beta_3\in I_3$ and $\beta_7\in I_7$, so that
   $v(r,\varphi)$ is negative on $(\beta_3,\beta_7)$, 
   and positive on $(\beta_7,\beta_3+2\pi)$.
   \end{lemma}

   \begin{proof}
   Observing that $v(r,\varphi)=2a_2r(\cos\varphi+\Orden(r))$, 
   the proof is similar to that of Lemma \ref{L9.13}. 
   \end{proof}

   \begin{lemma}
   There exists an $r_0>0$ such that for $0<r<r_0$
   the zeros of $u(r,\varphi)$ and $v(r,\varphi)$ satisfy the relation
   \begin{displaymath}
   \alpha_3<\beta_3,\qquad \beta_7<\alpha_7.
   \end{displaymath}
   \end{lemma}

   \begin{proof}
   Putting $a = -a_3/a_2>0$ we have for $0<r<r_0$ ($r_0$ small 
   enough to make  the previous lemmas valid)
   \begin{align*}
   u(r,\varphi)&=a_2r^2(\sin2\varphi-a r\sin3\varphi+\Orden(r^2)\\
   v(r,\varphi)&=2a_2r(\cos\varphi-\frac{3a}{2}r\cos2\varphi+\Orden(r^2)
   \end{align*}
   with $\Orden$-constants independent of $\varphi$. 

   The two zeros $\alpha_3$ and $\beta_3$ are on $I_3$ an interval with center 
   at $\frac{\pi}{2}$. At the point $\frac{\pi}{2}+ar$ we have
   \begin{align*}
   \frac{u(r, \pi/2+ar)}{a_2r^2}&=ar \cos(3ar)-\sin(2ar)+\Orden(r^2)\\
   \frac{v(r,\pi/2+ar)}{2a_2r}&=\frac{3ar}{2}\cos(2ar)-\sin(ar)+\Orden(r^2).
   \end{align*}
   Expanding in Taylor series we get
   \begin{align*}
   \frac{u(r, \pi/2+ar)}{a_2r^2}&=-ar+\Orden(r^2)\\
   \frac{v(r,\pi/2+ar)}{2a_2r}&=\frac{ar}{2}+\Orden(r^2).
   \end{align*}
   Choosing $r_0$ small enough we obtain $u(r, \pi/2+ar)<0<v(r, \pi/2+ar)$
   for $0<r<r_0$. 
   Since both $u(r,\varphi)$ and $v(r,\varphi)$ are decreasing on 
   this interval,  the zero of $u(r,\varphi)$ must come before 
   $\frac{\pi}{2}+ar$  and the zero of $v(r,\varphi)$ must come 
   after $\frac{\pi}{2}+ar$. 
   That is
   \begin{displaymath}
   \alpha_3<\frac{\pi}{2}+ar<\beta_3.
   \end{displaymath}

   The center of  $I_7$ is $\frac{3\pi}{2}$.  We compute the 
   functions at $\frac{3\pi}{2}-ar$. In the same way as before we find
   \begin{align*}
   \frac{u(r, 3\pi/2-ar)}{a_2r^2}&=-ar\cos(3ar)+\sin(2ar)+\Orden(r^2)
   =a r+\Orden(r^2)\\
   \frac{v(r,3\pi/2-ar)}{2a_2r}&=\frac{3ar}{2}\cos(2ar)-\sin(ar)+
   \Orden(r^2)=\frac{ar}{2}+\Orden(r^2).
   \end{align*}
   On the interval $I_7$ the function $u(r,\varphi)$ is decreasing whereas
   $v(r,\varphi)$ is increasing, so that the above computation implies that 
   for $r_0$ small enough, we will have that the zero of $u(r,\varphi)$ 
   will come after $\frac{3\pi}{2}-ar$, and that the zero of $v(r,\varphi)$ 
   will come before this value. 
   That is
   \begin{displaymath}
   \beta_7<\frac{3\pi}{2}-a r<\alpha_7.
   \end{displaymath}
   \end{proof}

   \begin{proof}[Proof of Proposition \ref{P1}]
   Taking $r_0$ small enough all  previous lemmas will apply. 
   We have seen that  the zeros of $u(r,\varphi)$ and $v(r,\varphi)$ 
   satisfy 
   \begin{displaymath}
   \alpha_1 < \alpha_3<\beta_3<\alpha_5<\beta_7<\alpha_7<\alpha_1+2\pi
   \end{displaymath}
   so that in particular these functions do not vanish simultaneously. 
   Therefore,  the curve $\gamma_r$ with equation 
   \begin{displaymath}
   \varphi\mapsto h(r,\varphi)=u(r,\varphi)+iv(r\varphi)
   \end{displaymath}
   does not pass through $z=0$. 

   Since we know the sign of $u$ and $v$ on the intervals limited by 
   the above zeros, we easily compute  the index $\omega(\gamma_r,0)=1$.
   \end{proof}

   \begin{theorem}
   Let $f$ be a holomorphic function in the conditions of 
   Proposition \ref{P1}.
   Let $(f_n)$ be a sequence of holomorphic functions on the 
   disc where $f$ is defined and converging uniformly to $f$ on 
   compact sets of this disc. 
   Then there exist $n_0$ and a sequence $(b_n)$ of complex numbers 
   such that for $n\ge n_0$,   $b_n$ is a turning point of $f_n$ and 
   $\lim_n b_n =0$.
   \end{theorem}

   \begin{proof}
   Let $r_0$ be small enough to make all previous lemmas applicable to $f$. 
   Put $u_n(r,\varphi):=\Imag f_n(re^{i\varphi})$ and $v_n(r,\varphi)=
   \Real f_n'(re^{i\varphi})$. The uniform convergence implies that  
   for each \mbox{$0<r<r_0$}, $\lim_n u_n(r,\varphi)=u(r,\varphi)$ and 
   $\lim_n v_n(r,\varphi)=v(r,\varphi)$ uniformly in $\varphi$. Finally put \break
   $h_n(r,\varphi):=u_n(r,\varphi)+iv_n(r,\varphi)$.

   Let $(r_n)$ be a decreasing sequence  of real numbers with $0<r_n<r_0$ and 
   $\lim_nr_n=0$.

   In  Proposition \ref{P1} $h(r_n,\varphi)$ does not vanish. Since it is 
   continuous there exists a $\delta_n>0$ such that 
   $|h(r_n, \varphi)|>\delta_n$ for all $\varphi$. 
   By the uniform convergence there exists $N_n$ such that 
   $|h(r_n, \varphi)-h_m(r_n, \varphi)|<\delta_n$
   for each $m\ge N_n$ and all $\varphi$. 

   Let $\gamma_n$ be the curve  $\varphi\mapsto h(r_n, \varphi)$.
   We have seen in Proposition \ref{P1} that $\omega(\gamma_n,0)=1$. 
   Let $\gamma_n^{(m)}$ be  the curve  $\varphi\mapsto h_m(r_n, \varphi)$. 
   Since 
   \begin{displaymath}
   |h(r_n, \varphi)-h_m(r_n, \varphi)|<\delta_n<|h(r_n, \varphi)|,
   \qquad  (m\ge N_n)
   \end{displaymath}
   we find that $\omega(\gamma^{(m)}_n,0)=\omega(\gamma_n,0)=1$. 

   Since $\omega(\gamma_n^{(m)},0)=1$ there is no homotopy of the curve 
   to a point in  $\C\smallsetminus \{0\}$. The equation of this 
   curve is 
   \begin{displaymath}
   \varphi\mapsto h_m(r_n,\varphi).
   \end{displaymath}
   The curves $\varphi\mapsto h_m(r,\varphi)$ for $0\le r\le r_n$ 
   will be a homotopy of $\gamma_n^{(m)}$ to the point $h_m(0,\varphi)$ 
   if this function does not vanish for 
   $(r,\varphi)\in[0,r_0]\times[0,2\pi]$. It follows that there is 
   a point with  $h_m(r,\varphi)=0$. This makes $b_{n,m}:=re^{i\varphi}$ 
   a turning point of $f_m$ with  $|b_{n,m}|\le r_n$

   For each $n$ we have found $N_n$ such that for  $m\ge N_n$ 
   there exists a turning point $b_{n,m}$ of $f_m$ with $|b_{n,m}|< r_n$. 
   It is clear that we may take  $N_1<N_2<N_3<\cdots $. 

   Now define for $N_k\le m<N_{k+1}$ the point $b_m:=b_{k,m}$.
   This is a sequence defined for $m\ge N_1$. 

   The sequence $(b_m)$ satisfies our theorem.  Indeed, by construction  
   $b_m$ is a turning point for $f_m$ and  for each $m$ there is a 
   $k$ with  $|b_m|=|b_{k,m}|<r_k$  where $N_k\le m<N_{k+1}$. Hence 
   for $m>N_k$ we will have $|b_m|<r_j\le r_k$, so that $\lim b_m = 0$.
   \end{proof}
   
   Now we can prove the last part of Theorem \ref{T9.1}:
   \emph{There is a sequence $(b_n)$ of turning points for 
   $\zeta(s)$ with $\lim_{n\to\infty} \Real(b_n)= E$.}

   \begin{proof}[Proof of the second half of Theorem \ref{T9.1}]
   Let $g(s):=\frac{2^s-1}{2^s+1}\zeta(s)$, and define $f(s)=g(s+E)$. 
   We then have $f(0)=0.933\dots$, $f'(0)=0$,  $f''(0)=0.070\dots$, 
   $f'''(0)=-0.178\dots$.

   By Lemma  \ref{KH}  there exists a sequence $(t_n)$ of real numbers with 
   \begin{displaymath}
   \lim_{n\to\infty}\zeta(s+it_n) = g(s)=f(s-E)
   \end{displaymath}
   uniformly on compact sets of $\sigma>1$.

   It follows that the functions $\zeta(s+E+it_n)$ converge to 
   $f(s)$ uniformly on the  disc with center $0$ and radius $E-1$. 

   By Theorem \ref{P1} there exists a sequence $(c_n)$ such that 
   $c_n$ is a turning point of $\zeta(s+E+it_n)$ and $\lim_n c_n=0$.

   Put $b_n = c_n+E+it_n$. It is clear that $b_n$ is a turning 
   point of $\zeta(s)$ and
   \begin{multline*}
   \lim_{n\to\infty}\Real (b_n)=\lim_{n\to\infty}\Real (c_n+E+it_n)=
   \lim_{n\to\infty}\Real (c_n+E)=\\
   =E+\lim_{n\to\infty}\Real (c_n)=E+\Real(
   \lim_{n\to\infty} c_n)=E.
   \end{multline*}
   \end{proof}

\end{document}